\documentclass[english]{amsart}%

\usepackage{amsmath}    
\usepackage{amssymb}    
\usepackage{amsthm}
\usepackage{amscd}
\usepackage{bbm}        
\usepackage{babel}
\usepackage{paralist}

\newcommand{\R}{\mathbbm{R}}
\newcommand{\N}{\mathbbm{N}}
\newcommand{\C}{\mathbbm{C}}

\newcommand{\Rn}{{\mathbbm{R}^n}}

\newcommand{\Sp}{\mathrm{Sp}}

\newcommand{\vol}{\mathrm{vol}}

\DeclareMathOperator*{\res}{\mathrm{res}}

\newcommand{\re}{\mathrm{Re}\,}

\def\dbar{{\mathchar'26\mkern-12mu d}}

\theoremstyle{plain} \newtheorem{theorem}{Theorem}[section]
\theoremstyle{definition} 
\theoremstyle{plain} 
\theoremstyle{plain} 
\theoremstyle{plain} \newtheorem{corollary}[theorem]{Corollary}
\theoremstyle{definition} \newtheorem{definition}[theorem]{Definition}
\theoremstyle{plain} \newtheorem{proposition}[theorem]{Proposition}
\theoremstyle{remark} \newtheorem{remark}[theorem]{Remark}
\theoremstyle{definition} \newtheorem{hypothesis}[theorem]{Hypothesis}
\theoremstyle{definition} \newtheorem{example}[theorem]{Example}
\theoremstyle{plain}

\begin{document}
\title[Heat trace asymptotics in Euclidean space]{Heat trace asymptotics of\\subordinate Brownian motion\\in Euclidean space}

\author{M.A.~Fahrenwaldt}
\address{Institut f\"ur Mathematische Stochastik, Leibniz Universit\"at Hannover, Welfengarten 1, 30167 Hannover, Germany}
\email{fahrenw@stochastik.uni-hannover.de}

\begin{abstract} 
We derive the heat trace asymptotics of the generator of subordinate Brownian motion on Euclidean space for a class of Laplace exponents. The terms in the asymptotic expansion can be computed to arbitrary order and depend both on the geometry of Euclidean space and the short-time behaviour of the process. If the Blumenthal-Getoor index of the process is rational, then the asymptotics may contain logarithmic terms. The key assumption is the existence of a suitable density for the L\'evy measure of the subordinator. The analysis is highly explicit.
\end{abstract}

\maketitle

%%%%%%%%%%%%%%%%%%%%%%%%%%%%%%%%%%%%%%%%%%%%%%%%%%%%%%%%%%%%%%%%%%%%%%%%%%%%
\section{Introduction}

This paper explores the correspondence between stochastic processes on $\R^n$ and analytical objects: we pick a subordinator $X_t$, i.e. an increasing process with independent and homogeneous increments, and let $B_{X_t}$ be the subordinate Brownian motion. It is obtained from a standard Brownian motion by introducing a local time given by the subordinator. Let $A$ be the generator of the corresponding semigroup. We show that in the heat trace asymptotics of $A$ we recover both geometric information about $\R^n$ and probabilistic information about $B_{X_t}$ and $X_t$. This is illustrated schematically in the following diagram and will be made precise in the next section. 

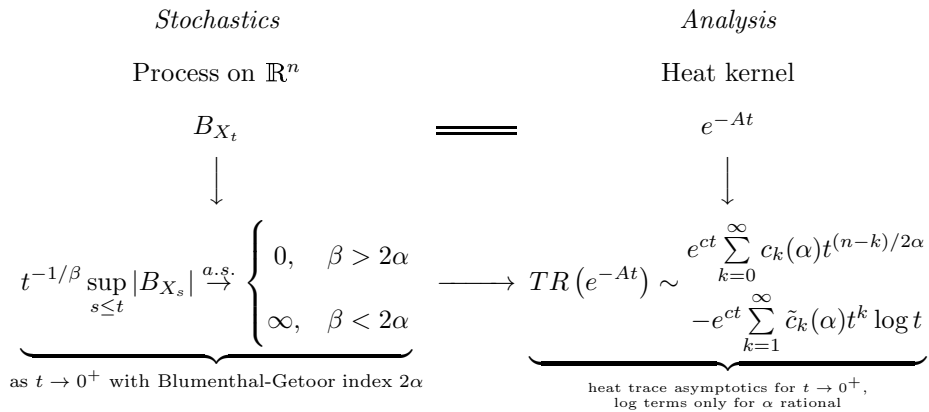
\[
\begin{CD}
\textit{Stochastics} @. \textit{Analysis} \\
\text{Process on } \R^n @. \text{Heat kernel} \\
B_{X_t} @= e^{- A t} \\
@VVV @VVV \\
\underbrace{
 t^{-1/\beta}\sup_{s \leq t} \left|B_{X_s}\right| \stackrel{a.s.}{\to} \left\{
\begin{matrix}
0, & \beta > 2 \alpha \\
\infty, & \beta < 2 \alpha  
\end{matrix} \right.
}_\text{as $t \to 0^+$ with Blumenthal-Getoor index $2 \alpha$} @>>> \underbrace{
TR\left(e^{- A t}\right) \sim 
\begin{matrix} 
e^{ct} \! \sum\limits_{k=0}^\infty c_k(\alpha) t^{(n-k)/2 \alpha} \\
-e^{ct} \! \sum\limits_{k=1}^\infty \tilde{c}_k(\alpha) t^k \log t
\end{matrix}
}_{\text{heat trace asymptotics for $t \to 0^+$,} \atop \text{log terms only for $\alpha$ rational}}
\end{CD}
\]

Three features differentiate our results from the literature. First, we consider processes that live in Euclidean space and not in compact domains. Second, we observe that the heat trace asymptotics strongly depend on the (ir)rationality of a parameter that governs the short-term behaviour of the process. Third, our method allows to compute the asymptotics to arbitrary order. 

We consider a class of subordinators that is small enough to allow a fully tractable analysis yet large enough to show interesting behaviour. Roughly speaking, we work with subordinators whose Laplace exponent in its L\'evy-Khintchin form has a density with suitable asymptotic expansion near the origin and is of rapid decay at infinity. This class contains the relativistic stable process, which is important in applications in financial mathematics or quantum physics. We present this as a fully worked example in dimensions 2 and 3 demonstrating that the latter situation leads to logarithmic terms in the heat trace asymptotics.

The essential technical features of our approach are the use of a global calculus of classical pseudodifferential operators on $\R^n$ and a generalized trace functional on this algebra. The trace functional allows us to go beyond compact domains and consider processes with values in the whole of $\R^n$.

More precisely, for our class of subordinators the suitably shifted generator $\tilde A=A-c I$ belongs to the algebra of classical $SG$-pseudodifferential operators on $\R^n$ \cite{egorov1997pseudo} \cite{maniccia2006complex} and so do its complex powers $\tilde A^{-z}$ and the heat operator $e^{-\tilde A t}$. Using a generalized trace functional $TR$ analogous to \cite{maniccia2013determinants} on a suitable subalgebra we explicitly compute the regularized zeta function $\zeta=TR(\tilde A^{-z})$ and the generalized heat trace $TR(e^{-\tilde A t})$. The asymptotics of the latter can be expressed in terms of the pole structure of the regularized zeta function. The trace functional $TR$ is the analogue of the Kontsevich-Vishik trace \cite{kontsevich1994determinants} \cite{kontsevich1995geometry} on closed manifolds which is used to study determinants of elliptic operators. These play an important role in many areas of mathematics and quantum physics, cf. \cite{scott2010traces} for comprehensive references. 

In a broader perspective, we argue similarly to noncommutative geometry \cite{connes1994noncommutative}: given a stochastic process we try to infer probabilistic information from the spectrum of a strategically associated operator. 

Closely related to this investigation are \cite{banuelos2014two} \cite{park2014trace} that compute several terms in the heat trace expansion for the relativistic $\alpha$-stable processes on a compact domain in Euclidean space. We comment on this in Section \ref{sec_example} and also show that our results agree with \cite{banuelos2014two} and \cite{park2014trace}, cf. Remark \ref{rem_banuelos_trace}. The case of subordinate Brownian motion on closed manifolds is covered in \cite{applebaum2011infinitely} \cite{applebaum2014probability} \cite{banuelos2013trace} \cite{banuelos2008trace}.   

The use of pseudodifferential operators to understand Feller processes is the theme of the comprehensive series \cite{jacob2001pseudo} \cite{jacob2002pseudo} \cite{jacob2005pseudo}. Building on the seminal work of Hoh-Jacob-Schilling, the author describes a symbolic calculus for a class of pseudodifferential operators that appear naturally as generators of certain Markov processes. These operators have a more intricate structure than the pseudodifferential operators used in index theory yet allow for parametrices and the Fredholm property. Also, the Ruzhansky-Turunen theory of pseudodifferential operators on Lie groups has found successful application in Markov processes as developed by Applebaum \cite{applebaum2011pseudo}. Either calculus appears, however, not suitable for our purposes due to the lack of a homogeneous symbol expansion. 

This paper is organized as follows. The following section recalls the basic notation for symbols of pseudodifferential operators. Section \ref{sec_key_results} contains the statements of our key results with proofs given in Section \ref{sec_proofs}. Section \ref{sec_example} contains a fully worked example also illustrating several calculations from the proofs.

%%%%%%%%%%%%%%%%%%%%%%%%%%%%%%%%%%%%%%%%%%%%%%%%%%%%%%%%%%%%%%%%%%%%%%%%%%%%%%%
\section{Symbol spaces} \label{sec_notation}
We briefly recall the symbol spaces of the theory of $SG$-pseudodifferential operators on $\R^n$, the reader is referred to \cite{cordes1995technique} \cite{egorov1997pseudo} \cite{maniccia2006complex} \cite{maniccia2013determinants} \cite{nicola2010global} for a more detailed treatment. For the notation to be consistent in this paper we need to adjust some of the standard notation in pseudodifferential operators.

We define the smooth weight function $\langle x \rangle = (1+|x|^2)^{1/2}$ and introduce the partial differential operators $\partial_j = \frac{\partial}{\partial x_j}$. For a multi-index $\beta=(\beta_1, \ldots, \beta_n) \in \N_0^n$ we set $\partial_x^\beta = \partial_1^{\beta_1} \cdots \partial_n^{\beta_n}$.

Denote by $C^\infty(\R^n)$ the set of smooth functions $\R^n \to \C$ and let $\mathcal S(\R^n)$ be the {\em Schwartz space} of rapidly decaying functions, i.e. $u \in C^\infty(\R^n)$ such that
\[
\sup_{x \in \R^n} \left| x^\gamma \partial_x^\beta u(x)\right| < \infty
\]
for all $\beta, \gamma \in \N_0^n$. Denote by $\mathcal S'(\R^n)$ the topological dual of the Schwartz space, called {\em temperate distributions}.

As usual, let $L^2(\R^n)$ be the Hilbert space $L^2(\R^n)=\left\{u \in \mathcal S'(\R^n) \middle| (u, u) < \infty\right\}$ with inner product $(u,v)= \int u(x) \overline{v(x)}dx$ and norm $||u||_{L^2(\R^n)}^2=(u,u)$.

Define the Fourier transform $\hat u(\xi)$ of a function $u \in \mathcal S(\R^n)$ by
\[
\hat u(\xi) = \int e^{i x \cdot \xi} u(x) \dbar x,
\]
where $\dbar x = (2 \pi)^{-n/2}dx$ and $\cdot$ denotes the standard inner product in $\R^n$. The Fourier transform yields an isomorphism of $\mathcal S(\R^n)$ extending to an isomorphism of $\mathcal S'(\R^n)$.
 
We now introduce the symbol class $S^{r, r'}(\R^n)$.

\begin{definition}[\cite{maniccia2006complex}, Definition 2.1] \label{def_symbol}
The space $S^{r, r'}(\R^n)$ of \emph{symbols} of order $(r, r') \in \C^2$ is the set of smooth functions $\sigma: \R^n \times \R^n \to \C$ such that for all multi-indices $\beta, \gamma \in \N^n_0$ there is a constant $C_{\beta, \gamma}$ with
\[
\left| \partial_\xi^\beta \partial_x^\gamma \sigma(x, \xi) \right| \leq C_{\beta, \gamma} \langle x \rangle^{\re r'-|\gamma|} \langle \xi \rangle^{\re r-|\beta|} 
\]
for any $(x, \xi) \in \Rn \times \Rn$. We set $S^{-\infty, r'}(\R^n)=\cap_{r \in \R} S^{r, r'}(\R^n)$.
\end{definition}

The pseudodifferential operator $A$ with symbol $\sigma \in S^{r, r'}(\R^n)$ is given by
\begin{equation} \label{def_pseudo}
Au (x) = \int e^{ix \cdot \xi} \sigma(x, \xi) \hat u(\xi) \dbar x
\end{equation}
mapping $\mathcal S(\R^n) \to \mathcal S(\R^n)$ and $\mathcal S(\R^n)' \to \mathcal S(\R^n)'$. 

For symbols independent of $x$ we recall symbol expansions and classicality of pseudodifferential operators. The classical symbols have asymptotic expansions in $\xi$ into homogeneous terms, we refer to Definition 2.2 of \cite{maniccia2006complex} for the general situation. 

\begin{definition} \label{def_classical_expansion}
Let $\sigma(\xi) \in S^{r,0}(\R^n)$ be a symbol independent of $x$.
\begin{enumerate}[(i)]
\item We say that it has the asymptotic expansion
\[
\sigma(\xi) \sim \sum_{k=0}^\infty \sigma_{r-k}(\xi)
\]
where $\sigma_{r-k} \in C^\infty(\R^n \setminus \{0\})$ if the following holds. For a cutoff function $\chi \in C^\infty(\R^n)$ equal to 1 outside $|\xi| \geq 1$ and equal to zero for $|\xi| \leq 1/2$ we have
\begin{equation*} 
\sigma(\xi) - \sum_{k=0}^{N-1} \chi(\xi) \sigma_{r-k}(\xi) \in S^{r-N, r'}(\R^n)
\end{equation*}
for any $N \geq 1$. 
\item The symbol is \emph{classical} if the $\sigma_{r-k}$ are homogeneous in $\xi$, i.e. if $\sigma_{r-k}(\tau \xi)= \tau^{r-k} \sigma_{r-k}(\xi)$ for $\tau > 0$.
The set of classical symbols is denoted by $S^{r,0}_{cl}(\R^n)$.
\end{enumerate}
\end{definition}

We also recall the notion of ellipticity and parameter-dependent ellipticity in the context of symbols that are independent of $x$. We refer to \cite{maniccia2006complex}, Definition 2.3 and \cite{maniccia2013determinants}, A.3, respectively, for the full generality of the definitions.

\begin{definition} \label{def_elliptic_parameter_elliptic}
Let $\sigma(\xi) \in S^{r, 0}(\R^n)$ be a symbol independent of $x$ with order $r>0$.

\begin{enumerate}[(i)]
\item \label{def_elliptic}
We say that $\sigma$ is \emph{elliptic} if there is a constant $C \geq 0$ such that
\[
C \langle \xi \rangle^r \leq \sigma(\xi)
\]
for all $\xi \in \R^n$.

\item \label{def_parameter_elliptic}
Fix $0<\theta<\pi$ and let $\Lambda=\{r e^{i \varphi} | r \geq 0 \textrm{ and } \theta \leq \varphi \leq 2\pi-\theta\}$ be a sector in the complex plane. We say that $\sigma$ is \emph{$\Lambda$-elliptic} if there is a constant $C \geq 0$ such that 
\begin{enumerate}[(a)]
\item $\sigma(\xi)$ does not take values in $\Lambda$ for all  $\xi \in \R^n$, and 
\item $|(\lambda-\sigma(\xi))^{-1}| \leq C \langle \xi \rangle^{-r}$ for $\lambda \in \Lambda$ and $\xi \in \R^n$.
\end{enumerate}

\end{enumerate}
\end{definition}

%%%%%%%%%%%%%%%%%%%%%%%%%%%%%%%%%%%%%%%%%%%%%%%%%%%%%%%%%%%%%%%%%%%%%%%%%%%%%%%
\section{Statement of the key results} \label{sec_key_results}
We state and motivate the assumptions and formulate the key results.\\

\textbf{Probabilistic prelude}, cf. \cite{bertoin1998levy} \cite{jacob2001pseudo}. Let $B_t$ be a Brownian motion on $\Rn$ with characteristic function $\mathbbm E(e^{i \xi \cdot B_t})=e^{-t |\xi|^2}$ for $\xi \in \R^n$ and $t>0$. Let $X_t$ be a subordinator on $[0, \infty)$ independent of $B_t$, i.e. an increasing L\'evy process with values in $[0, \infty)$ and $X_0=0$ almost surely. The distribution of $B_{X_t}$ can be described in terms of the characteristic function $\mathbbm E(e^{i \xi \cdot B_{X_t}})=e^{-t f(|\xi|^2)}$ for $t>0$, where $f$ is the \emph{Laplace exponent} in the probabilist's convention or the \emph{Bernstein function} in the analyst's vocabulary. Also the generating function of the subordinator is $\mathbbm E(e^{- \lambda X_t}) = e^{- t f(\lambda)}$ for $\lambda>0$. 

Recall (cf. \cite{schilling2012bernstein}, Definition 3.1) that a function $f: (0, \infty) \to \R$ is a \emph{Bernstein function} if $f$ is smooth, $f(\lambda) \geq 0$ and $(-1)^{k-1} f^{(k)}(\lambda) \geq 0$ for $k \in \N$.  Any Bernstein function can be represented in L\'evy-Khintchin form as
\begin{equation} \label{eq_Bernstein_general}
f(\lambda)=a + b \lambda + \int_0^\infty \left(1-e^{-\lambda t}\right) \mu(dt),
\end{equation}
for constants $a, b \geq 0$ and $\mu$ a measure on $(0,\infty)$ such that $\int_0^\infty t \wedge 1 \; \mu(dt) < \infty$. The L\'evy characteristic triplet $(a, b, \mu)$ uniquely determines $f$. 

We consider Bernstein functions whose L\'evy measure has a locally integrable density $m$ with respect to Lebesgue measure, and we call this the \emph{L\'evy density}. Moreover, we restrict ourselves to Bernstein functions of the form \eqref{eq_Bernstein_general} with $a=b=0$ but indicate in Section \ref{sec_proofs} how the general case differs.

One can canonically associate a semigroup $T_t$ with the process $B_{X_t}$ by defining $[T_t u](x)=\mathbbm E^x(u(B_{X_t}))$ acting on Schwartz functions $u \in \mathcal S(\Rn)$. The generator of this semigroup is defined as the operator $Au = \lim_{t \to 0} \frac{T_t u - t}{t}$ with domain the set of functions where this limit exists. The link with the probabilistic picture is that the generator acts as the integral operator
\begin{equation} \label{eq_psido_symbol_f}
A u(x) = - (2 \pi)^{-n/2} \int_{\Rn} e^{i x \cdot \xi} f\left(|\xi|^2\right) \hat{u}(\xi) d \xi,
\end{equation}
where $u \in \mathcal S(\Rn)$ and $\hat{u}$ denotes the Fourier transform of $u$. So if $f$ is smooth and its derivatives decay sufficiently fast, then $A$ is a pseudodifferential operator with symbol $-f(|\xi|^2)$. We sometimes call this the symbol of the subordinate process.

We recall the definition of asymptotic expansions of real-valued functions.
\begin{definition}
Suppose that $g: (0,\infty) \to \R$ is a function. We say that $g(t) \sim \sum_{k=0}^\infty p_k t^{a_k}$ as $t \to 0^+$ if $p_k \in \R, a_k \uparrow \infty$ and
\[
\lim_{t \to 0^+} t^{-a_N}\left(g(t)-\sum_{k=0}^{N} p_k t^{a_k}\right) = 0
\]
for every $N \geq 0$. Analogously for $t \to \infty$.
\end{definition}

The key assumption in this paper is the existence of a suitable L\'evy density.

\begin{hypothesis} \label{hypo_main_hypothesis}
Let $f(\lambda)=\int_0^\infty \left(1-e^{-\lambda t}\right) m(t) dt$ be a Bernstein function with locally integrable density $m: (0, \infty) \to \R$ that has the following properties.
\begin{enumerate}[(i)]
\item There is an $\alpha \in (0,1)$ such that $m$ has the asymptotic expansion 
\[
m(t) \sim t^{-1-\alpha} \sum_{k=0}^\infty p_k t^k
\]
as $t \to 0^+$. 
\item $m$ is of rapid decay at $\infty$, i.e. $m(t) t^\beta$ is bounded a.e. for $t>1$ for all $\beta \in \R$.
\item $\overline{m}(0, \infty)<0$ where $\overline{m}(0, \infty)=\int_0^\infty \left(m(t)-p_0 t^{-1-\alpha}\right) dt$.
\end{enumerate}
\end{hypothesis}

Assumption (i) yields an asymptotic expansion of $f$ for large $\lambda$ and assumption (ii) makes $f$ smooth at the origin as it implies that $\int_0^\infty t^l m(t) d t < \infty$ for any $l \in \N$. Moreover, (iii) is a technical condition ensuring $A-\overline{m}(0, \infty)$ with $A$ as in \eqref{eq_psido_symbol_f} is an invertible and classical pseudodifferential operator.

\begin{example}
From \cite{schilling2012bernstein} we pick five examples of Bernstein functions that satisfy Hypothesis \ref{hypo_main_hypothesis}. The asymptotics of the L\'evy densities can be obtained using Taylor's theorem. In each case, $\alpha \in (0,1)$ and $c>0$.

\begin{enumerate}[(i)]
\item $f(\lambda)=(\lambda+1)^\alpha-1$. This has L\'evy density $m(t)=\frac{\alpha}{\Gamma(1-\alpha)} e^{-t} t^{-\alpha-1}$ and the asymptotic expansion
\[
m(t) \sim \tfrac{\alpha}{\Gamma(1-\alpha)} t^{-1-\alpha}\left(1-t+\tfrac{1}{2}t^2+ \cdots\right)
\]
as $t \to 0^+$. This Bernstein function describes the relativistic $\alpha$-stable L\'evy processes, which is related to the relativistic Hamiltonian in physics, cf. \cite{albeverio2001analytic} and to the Normal Inverse Gaussian distribution used in financial mathematics, cf. \cite{barndorff1997processes}. All key results are illustrated for the case $\alpha=1/2$ in Section \ref{sec_example}. 

\item $f(\lambda)=\lambda / (\lambda+c)^\alpha$ with $m(t)=\tfrac{\sin(\alpha \pi) \Gamma(1-\alpha)}{\pi} e^{-at} t^{\alpha-2} (c t+1-\alpha)$ and asymptotics
\[
m(t) \sim \tfrac{\sin((1-\alpha') \pi) \Gamma(\alpha')}{\pi} t^{-1-\alpha'}\left(\alpha'+c(1-\alpha') t +c^2\left(\tfrac{1}{2} \alpha'^2-1\right)t^2+ \cdots\right)
\]
where $\alpha'=1-\alpha$. 
\item $f(\lambda)=\lambda\left(1-e^{-2 \sqrt{\lambda+c}}\right) / \sqrt{\lambda+c}$ and $m(t)=\frac{e^{-1/t-c t}\left(2+t(e^{1/t}-1)(1+2ct)\right)}{2 \sqrt{\pi} t^{5/2}}$. Then
\[
m(t) \sim \tfrac{1}{2 \sqrt{\pi}} t^{-1-1/2}\left(1+ct-\tfrac{3c^2}{2}t^2+\cdots\right).
\]

\item $f(\lambda)=\Gamma\left(\frac{\lambda+c}{2c}\right) / \Gamma\left(\frac{\lambda}{2c}\right)$ with $m(t)=\frac{c^{3/2} e^{2ct}}{2 \sqrt{\pi}(e^{2ct}-1)^{3/2}}$. Here, 
\[
m(t) \sim \tfrac{1}{\sqrt{32 \pi}} t^{-1-1/2}\left(1+\tfrac{1}{2}ct-\tfrac{1}{8}c^2t^2 + \cdots \right).
\]
\item $f(\lambda)=\Gamma(\alpha \lambda+1) / \Gamma(\alpha \lambda+1-\alpha)$. This has density $m(t)=\frac{e^{-t/\alpha}}{\Gamma(1-\alpha)(1-e^{-t/\alpha})^{1+\alpha}}$ with asymptotics
\[
m(t) \sim \tfrac{\alpha^{1+\alpha}}{\Gamma(1-\alpha)} t^{-1-\alpha}\left(1+\tfrac{\alpha-1}{2 \alpha}t + \tfrac{3 \alpha^2-7\alpha+2}{24 \alpha^2}t^2+\cdots\right).
\]
\end{enumerate}
All of the above densities are of exponential decay; we are not aware of densities that are of rapid but not exponential decay.  
\end{example}

\textbf{Sample path properties.} We next state the relationship between the asymptotic expansion of the L\'evy density and sample path properties of the subordinate Brownian motion $B_{X_t}$ and the subordinator $X_t$. 

A link with pathwise properties of the subordinate Brownian motion can be established via the Blumenthal-Getoor index that concerns the short-time behaviour of $B_{X_t}$. This result is a simple consequence of \cite{blumenthal1961sample} \cite{pruitt1981growth} \cite{schilling1998growth}.

\begin{theorem} \label{thm_short_long_time}
Let $B_{X_t}$ be the subordinate Brownian motion with subordinator $X_t$ whose Bernstein function satisfies Hypothesis \ref{hypo_main_hypothesis}. Then 
\[
\lim_{t \to 0} \; t^{-1/\beta} \sup_{s \leq t} \left| B_{X_s} \right| = 
\left\{ \begin{matrix}
0 & \text{for all } \beta> 2 \alpha \\
\infty & \text{for all } \beta < 2 \alpha,
\end{matrix} \right.
\]
with probability one.
\end{theorem}

The relationship with the subordinator is via L\'evy's arcsine law. One can express the order $\alpha$ and the coefficients $p_0, p_1, \ldots$ as expectations of suitable random variables: the asymptotics of $m$ near $t=0$ thus have a probabilistic representation. 

\begin{theorem} \label{thm_probab_interpretation}
Let the Bernstein function $f$ satisfy Hypothesis \ref{hypo_main_hypothesis} and let $X_t$ be the  corresponding subordinator. For $x>0$ define the first passage time strictly above $x$ by $T(x)=\inf\left\{ t \geq 0 | X_t > x \right\}$. Then
\[
\alpha = \lim_{x \to 0^+} \tfrac{1}{x} \mathbbm E\left(X_{T(x)^-}\right).
\]
The lowest-order coefficient is given as
\[
p_0 =  \tfrac{1}{\Gamma(-\alpha)} \tfrac{1}{t} \lim_{\lambda \to \infty} \lambda^{-\alpha} \log \mathbbm E\left(e^{- \lambda X_t} \right)
\]
for fixed $t>0$ with similar expressions for the higher-order coefficients. 
\end{theorem}

\textbf{Pseudodifferential operators and trace functional.} We work in the algebra of $SG$-operators (sometimes known as scattering operators) as summarized in \cite{maniccia2006complex} \cite{maniccia2013determinants}, cf. also Section \ref{sec_notation} for symbol spaces. 

We define the \emph{regularized trace functional} $TR$ exactly as in Section 2 of \cite{maniccia2013determinants} with the exception that there is no integration with respect to $x$ but only with respect to $\xi$, cf. equation \eqref{eq_defn_TR} for an explicit expression.  

A suitably shifted version of the generator $A$ and the corresponding heat operator are classical pseudodifferential operators. 

\begin{theorem} \label{thm_generator_as_pseudo}
Suppose $f$ is a Bernstein function satisfying Hypothesis \ref{hypo_main_hypothesis}. Define the operator $A$ as in \eqref{eq_psido_symbol_f} and
set $\tilde{A}=A-\overline{m}(0, \infty) I$. Define coefficients
\begin{equation} \label{eq_coefficients}
\alpha_k = -\Gamma(-\alpha+k) p_k 
\end{equation}
for $k=0, 1, 2, \ldots$ Then the following holds:
\begin{enumerate}[(i)]
\item The operator $\tilde{A}$ is a classical elliptic pseudodifferential operator whose symbol $\sigma(\tilde A)$ is in the class $S^{2 \alpha,0}_{cl}(\Rn)$ and has the asymptotic expansion
\[
\sigma\left(\tilde{A}\right)(\xi) \sim \sum_{k=0}^\infty \alpha_k |\xi|^{2\alpha-2k}
\]
in the sense of Definition \ref{def_classical_expansion}.
\item The heat operator $e^{-t \tilde{A}}$ is a pseudodifferential operator whose symbol $\sigma(e^{t\tilde{A}})$ belongs to $S_{cl}^{-\infty, 0}(\Rn)$ with asymptotic expansion
\[
\sigma(e^{t\tilde{A}})(\xi) \sim e^{-t \alpha_0 |\xi|^{2 \alpha}} - \left[\alpha_1 |\xi|^{2 \alpha-2}+\alpha_2 |\xi|^{2 \alpha-4} \right] t e^{-t \alpha_0 |\xi|^{2 \alpha}} \pm \cdots
\]
in the sense of Definition \ref{def_classical_expansion}.
\end{enumerate}
\end{theorem}

\textbf{The regularized zeta function.} We define the \emph{regularized zeta function} $\zeta(z)=TR\left(\tilde{A}^{-z}\right)$ where the complex powers of $\tilde{A}$ are defined by functional calculus in the $SG$-operators \cite{maniccia2006complex}. 

\begin{theorem} \label{thm_zeta_noninteger_order}
Under the assumptions of Theorem \ref{thm_generator_as_pseudo} and with $\tilde{A}=A-\overline{m}(0, \infty) I$, the function $\zeta(z)=TR\left(\tilde{A}^{-z}\right)$ is meromorphic on $\C$ with at most simple poles at the points $z_k=(n-k)/2 \alpha$ for $k=0,1,2, \ldots$. The point $z_n=0$ is a removable singularity. 
\end{theorem}

\begin{remark}
The present exposition uses zeta functions as an intermediate step towards the heat trace. However, operator zeta functions are important in their own right in diverse branches of mathematics and physics, cf. \cite{scott2010traces}. We also mention for completeness that the residues of the regularized zeta function are expressed in terms of the noncommutative residue. This was originally defined by Wodzicki on closed manifolds extending of the work of Adler-Manin  (cf. \cite{wodzicki1987noncommutative}) and plays a key role in noncommutative geometry. The residue depends only on the homogeneous component of order $-n$ in the symbol expansion of $\tilde A$.
\end{remark}

In the lowest orders, this residue becomes
\begin{equation} \label{eq_residues_explicitly}
\left. \begin{aligned}
\res_{z=z_0} \zeta(z) & = \tfrac{1}{2 \alpha} \tfrac{\Omega_n}{(2\pi)^{n}}  \alpha_0^{-n/2 \alpha} \\
\res_{z=z_1} \zeta(z) & = 0 \\
\res_{z=z_2} \zeta(z) & =  - \tfrac{1}{2 \alpha} \tfrac{\Omega_n}{(2\pi)^{n}} \alpha_0^{-z_2-1} \alpha_1 z_2 \\
\res_{z=z_3} \zeta(z) & = 0 \\
\res_{z=z_4} \zeta(z) & =  \tfrac{1}{2 \alpha} \tfrac{\Omega_n}{(2\pi)^{n}} \tfrac{1}{2} \alpha_1^2 \alpha_0^{-z_4-2} z_4 \left[z_4 + \tfrac{\alpha_1^2-2 \alpha_0 \alpha_2}{\alpha_1^2} \right] , 
\end{aligned} \right\rbrace
\end{equation}
with $\Omega_n=\tfrac{2 \pi^{n/2}}{\Gamma(n/2)}$ the surface area of the unit sphere in $\R^{n}$ and the $\alpha_k$ as in \eqref{eq_coefficients}. 

The location and the residues of these poles are determined in terms of the asymptotic expansion of $m$ near $t=0$. By Theorems \ref{thm_short_long_time} and \ref{thm_probab_interpretation} we can express this information probabilistically in terms of the subordinator. \\

\textbf{Heat trace expansion.} The pole structure of the regularized zeta-function determines the short-time asymptotics of the \emph{generalized heat trace} $TR(e^{-t \tilde{A}})$. 

\begin{theorem} \label{thm_heat_trace_noninteger_order}
Under the assumptions of Theorem \ref{thm_generator_as_pseudo} and with $\tilde{A}=A-\overline{m}(0, \infty) I$, the asymptotics as $t \to 0^+$ of the generalized heat trace $TR(e^{-t \tilde{A}})$ are given as follows.
\begin{enumerate}[(i)]
\item If $\alpha$ is rational, there are constants $c_k$ and $\tilde{c}_k$ such that
\[
TR(e^{-t \tilde{A}}) \sim \sum_{k=0}^\infty c_k t^{-(n-k)/2 \alpha} - \sum_{k=1}^\infty \tilde{c}_k t^k \log t.
\]
\item If $\alpha$ is irrational, we have 
\[
TR(e^{-t \tilde{A}}) \sim \sum_{k=0}^\infty c_k t^{-(n-k)/2 \alpha} 
\]
where $c_k = \Gamma\left(\frac{n-k}{2 \alpha}\right) \res\limits_{z=(n-k)/2 \alpha} \zeta(z)$.
\end{enumerate}
\end{theorem}

Since $\tilde A$ is merely $A$ shifted by a constant we note the immediate
\begin{corollary} \label{cor_heat_trace_A}
Under the assumptions of Theorem \ref{thm_heat_trace_noninteger_order} the heat trace expansion of $A$ is given as follows.
\begin{enumerate}[(i)]
\item $\alpha$ rational: there are constants $c_k$ and $\tilde{c}_k$ such that
\[
TR(e^{-t A}) \sim e^{- \overline{m}(0, \infty) t} \sum_{k=0}^\infty c_k t^{-(n-k)/2 \alpha} - e^{- \overline{m}(0, \infty) t} \sum_{k=1}^\infty \tilde{c}_k t^k \log t.
\]

\item $\alpha$ irrational: we have 
\[
TR(e^{-t A}) \sim e^{- \overline{m}(0, \infty) t} \sum_{k=0}^\infty c_k t^{-(n-k)/2 \alpha} 
\]
for constants $c_k$.
\end{enumerate}
\end{corollary}

Note the strikingly different behaviour for $\alpha$ rational and irrational with the appearance of logarithmic terms. For dimension $n>2$ we explicitly give the lowest-order terms (the case $n=2$ is illustrated in Section \ref{sec_example}):
\begin{equation} \label{eq_lowest_coeffs_heat_trace}
\left.\begin{aligned}
c_0 & =\tfrac{\Gamma(n/ 2\alpha)}{2 \alpha} \tfrac{\Omega_n}{(2\pi)^{n}}  \alpha_0^{-n/2 \alpha} \\
c_1 & =0 \\
c_2 & =-\tfrac{\Gamma\left((n-2)/2 \alpha \right)}{2 \alpha} \tfrac{\Omega_n}{(2\pi)^{n}} \alpha_0^{-(n-2)/2 \alpha-1} \alpha_1 \cdot \tfrac{n-2}{2 \alpha},
\end{aligned}\right\}
\end{equation}
where the $\alpha_k$ are as in \eqref{eq_coefficients}. The lowest-order term $c_0$ also appears in \cite{banuelos2014two} and \cite{park2014trace}, cf. Remark \ref{rem_banuelos_trace}.
The lowest-order non-logarithmic terms of the asymptotic expansion read
\[
\underbrace{\tfrac{\Omega_n}{(2\pi)^{n}}}_{\text{geometry} \atop \text{of $\R^n$}} \cdot
\underbrace{\tfrac{1}{2 \alpha} \left[ \Gamma(\tfrac{n}{2 \alpha}) \alpha_0^{-n/2 \alpha} t^{-n/2 \alpha} -\Gamma\left(\tfrac{n-2}{2 \alpha}\right) \alpha_0^{-(n-2)/2 \alpha-1} \alpha_1 \tfrac{n-2}{2 \alpha} t^{-(n-2)/2 \alpha} + \cdots \right]}_\text{dimension $n$ and probabilistic properties of $B_{X_t}$}.
\]

Note that the terms in the heat trace expansion combine the two main aspects of our subordinated Brownian motion: the flat Euclidean geometry in which the process lives and information about the L\'evy density.

\begin{remark}
We briefly comment on extensions and limitations of our approach.
\begin{enumerate}[(i)]
\item The approach works whenever the generator $A$ (or a suitably shifted generator) belongs to an algebra of classical pseudodifferential operators that has a trace functional and contains the complex powers $A^{-z}$ and the heat operator $e^{-At}$. Thus, on a closed manifold $M$ the same calculations can be carried out demonstrating how both the classical heat invariants and  certain probabilistic information of the subordinator appear explicitly in the heat trace asymptotics.  
\item The definition of the regularized trace $TR$ and the meromorphic extendibility of $\zeta(z)$ require classical symbols. This limits the class of Markov processes that can be treated by this method.
\item Since the symbol calculus of pseudodifferential operators works only modulo smoothing operators (symbols of order $-\infty$), the asymptotic expansion of $\sigma(A)$ cannot see the behaviour of the L\'evy density $m(t)$ for values of $t$ away from 0. Thus, one would have to work in a different operator algebra if one wanted to capture the behaviour of $m$ at both $t=0$ and $t=\infty$.
\end{enumerate}
\end{remark}

The present exposition naturally leads to further questions that are beyond the scope of the present exposition and are the subject of further research.
\begin{enumerate}[(i)]
\item Is there a probabilistic characterization, e.g. in terms of sample path properties, of our class of Bernstein functions?
\item How can one recover the "long end" of the asymptotics of $m$, i.e. as $t \to \infty$ in the spectrum of $A$?
\item What is the probabilistic significance of the logarithmic terms in the heat trace and the dichotomy rational/irrational $\alpha$?
\end{enumerate}

%%%%%%%%%%%%%%%%%%%%%%%%%%%%%%%%%%%%%%%%%%%%%%%%%%%%%%%%%%%%%%%%%%%%%%%%%%%%%
\section{Proofs of the key results} \label{sec_proofs}
This section contains the building blocks needed for the proof of the key results. We have structured the arguments into the following subsections.

\begin{enumerate}
\item Growth, regularity and asymptotics of a class of Bernstein functions
\item Sample path properties
\item Constructing a classical pseudodifferential operator
\item The zeta function and heat trace of the shifted operator
\end{enumerate}

%--------------------------------------------------------------------------------------------------------------------------------------
\subsection{Growth, regularity and asymptotics of a class of Bernstein functions}
We derive lower and upper bounds for our class of Bernstein functions, establish their smoothness and give explicit asymptotics.

As a general remark, note that by Theorem 3.9.29 of \cite{jacob2001pseudo}, any Bernstein function can be continuously extended to $[0, \infty)$ and we tacitly assume that all our functions are defined on this larger interval.

The following proposition suggests that we can construct a pseudodifferential operator with elliptic symbol given in terms of $f$ with an asymptotic expansion into homogeneous terms. This will be taken up in detail in the next subsection.

\begin{proposition} \label{prop_growth}
Let $f$ be a Bernstein function satisfying Hypothesis \ref{hypo_main_hypothesis}. Then the following holds.
\begin{enumerate}[(i)]
\item We have $f \in C^\infty([0,\infty))$, i.e. $f$ is smooth at the origin. \label{ass_smoothness}

\item The function $f$ has an asymptotic expansion
\begin{equation} \label{eq_asymptotic_f}
f(\lambda) \sim \overline{m}(0,\infty)-\sum_{k=0}^\infty \Gamma(-\alpha+k) p_k \lambda^{\alpha-k}
\end{equation}
as $\lambda \to \infty$ where $\overline{m}(0, \infty)=\int_0^\infty \left(m(t)-p_0 t^{-1-\alpha}\right) dt$. \label{ass_expansion}

\item The derivative $f^{(l)}$ has an asymptotic expansion
\begin{equation} \label{eq_asymptotic_f_derivative}
f^{(l)}(\lambda) \sim (-1)^{l+1} \sum_{k=0}^\infty \Gamma(-\alpha+l+k) p_k \lambda^{\alpha-l-k}
\end{equation}
as $\lambda \to \infty$ for any $l=1,2, \ldots$ \label{ass_expansion_derivative}

\item There is an $R>0$ such that
\[
- \tfrac{1}{2} \Gamma(-\alpha) p_0 \lambda^\alpha \leq f(\lambda)
\]
for $\lambda>R$. \label{ass_bound_for_ellipticity}

\item There are constants $C_l$ with $|f^{(l)}(\lambda)| \leq C_l \lambda^{\alpha-l}$ for $\lambda \geq 1$ and $l=0,1,2,\ldots$ \label{ass_derivatives}

\end{enumerate}
\end{proposition}

\begin{remark}
If we allow more general Bernstein functions of the form $f(\lambda)=a+b\lambda+\int_0^\infty (1-e^{-\lambda t}) m(t) dt$ with $a\geq 0$ and $b >0$, then the expansion in (ii) has the additional terms $a+b \lambda$. Also, the lower bound in (iii) reads $\tfrac{1}{2}b \lambda$.
\end{remark}

The key argument for the proof of assertions (\ref{ass_expansion}) and (\ref{ass_expansion_derivative}), from which  (\ref{ass_bound_for_ellipticity}) and  (\ref{ass_derivatives}) follow immediately, is a version of Watson's Lemma that we state without proof.

\begin{proposition}[\cite{bleistein1975asymptotic}, Section 4.1] \label{prop_Watson}
Let $g: (0, \infty) \to \R$ be a bounded and locally integrable function such that as $t \to 0^+$ one has
\[
g(t) \sim \sum_{k=0}^\infty p_k t^{a_k}.
\]
Assume $p_k \in \R$, $a_k>-1$ for all $k$ and $a_k \to \infty$ monotonically as $k \to \infty$. Then
\[
\int_0^\infty e^{- \lambda t} g(t) dt \sim \sum_{k=0}^\infty \Gamma(1+a_k) p_k \lambda^{-a_k-1}
\]
as $\lambda \to \infty$.
\end{proposition}

\begin{proof}[Proof of Proposition \ref{prop_growth}]
(\ref{ass_smoothness}) The smoothness of $f$ is a consequence of the rapid decay of $m$ as $t \to \infty$ since then $\int_0^\infty t^l m(t) dt < \infty$ for $l=1, 2, \ldots$. By Theorem 3.9.23 of \cite{jacob2001pseudo} the limit $f(0)=\lim_{t \to 0^+} f(t)$ exists so that $f$ is continuous at the origin. A direct calculation interchanging differentiation and integration shows that $f$ is differentiable on the half-line $[0, \infty)$. Moreover, any derivative of $f$ on $(0, \infty)$ (given by differentiating the Laplace transform of $m$) can be continuously extended to 0 as all higher moments of $m$ exist due to its rapid decay. We can repeat these arguments for any derivative of $f$ so that $f$ is smooth on $[0, \infty)$.

(\ref{ass_expansion}) This follows from Watson's Lemma. We consider
\[
\int_0^\infty \left(1-e^{-\lambda t}\right) m(t) dt = 
\int_0^\infty \left(1-e^{-\lambda t}\right) p_0 t^{-1-\alpha} dt + \int_0^\infty \left(1-e^{-\lambda t}\right) \overline{m}(t) dt
\]
with $\overline{m}(t)=m(t)-p_0 t^{-1-\alpha}$. The first integral can be evaluated and yields
\begin{equation} \label{eq_lowest_term}
\int_0^\infty \left(1-e^{-\lambda t}\right) p_0 t^{-1-\alpha} dt = \tfrac{\Gamma(1-\alpha)}{\alpha} p_0 \lambda^\alpha = -\Gamma(-\alpha) p_0 \lambda^\alpha.
\end{equation}
The second integral can be expanded using Watson's lemma where we have $g(t)=\overline{m}(t)$, $a_0=0$ and $a_k=-1-\alpha+k$ for $k \geq 1$. This yields
\begin{align}
\int_0^\infty \left(1-e^{-\lambda t}\right) \overline{m}(t) dt & = \int_0^\infty \overline{m}(t) dt - \int_0^\infty e^{-\lambda t} \overline{m}(t) dt \nonumber \\
& \sim \overline{m}(0, \infty) - \sum_{k=1}^\infty \Gamma(-\alpha+k) p_k \lambda^{\alpha-k} \label{eq_ass_m_bar}
\end{align}
for $\lambda \to \infty$. Adding \eqref{eq_lowest_term} and \eqref{eq_ass_m_bar} proves the assertion.

(\ref{ass_expansion_derivative}) Note that any derivative of $f$ is given in the form
\[
f^{(l)}(\lambda)=(-1)^{l+1} \int_0^\infty e^{-\lambda t} t^l m(t) dt.
\]
By Hypothesis \ref{hypo_main_hypothesis}, the map $t^l m(t)$ is integrable for any $l \geq 1$: at $t=0$ it grows like $t^{l-1-\alpha}$ and at infinity it is of rapid decay. So by Watson's Lemma 
\[
f^{(l)}(\lambda) \sim (-1)^{l+1} \sum_{k=0}^\infty \Gamma(-\alpha+l+k) p_k \lambda^{\alpha-l-k}
\]
for $\lambda \to \infty$.

(\ref{ass_bound_for_ellipticity}) The bound follows from the lowest-order asymptotics just derived. We find
\[
\lim_{\lambda \to \infty} \lambda^{-\alpha} f(\lambda) = -\Gamma(-\alpha) p_0.
\]
This means that for every $\epsilon>0$ there is some $R>0$ such that
$\left|\lambda^{-\alpha} f(\lambda) -p_0\right| < \epsilon$
for all $\lambda>R$. Now set $\epsilon=-\Gamma(-\alpha)p_0/2$ so that 
\begin{equation} \label{eq_upper_lower_bound}
-\tfrac{1}{2} \Gamma(-\alpha) p_0 < \lambda^{-\alpha} f(\lambda) < -\tfrac{3}{2} \Gamma(-\alpha) p_0
\end{equation}
for all $\lambda>R$. This proves the claim.

(\ref{ass_derivatives}) The bound for $l=0$ follows from \eqref{eq_upper_lower_bound}. The claim for $l \geq 1$ follows by considering $\lim_{\lambda \to \infty} \lambda^{-\alpha+l} f^{(l)}(\lambda)=(-1)^{l+1} \Gamma(-\alpha+l) p_0$ by (\ref{ass_expansion_derivative}) and arguing as in (\ref{ass_bound_for_ellipticity}). One can have $R=1$ by choosing $C_l$ sufficiently large.

\end{proof}

%---------------------------------------------------------------------------------------------------------------------
\subsection{Sample path properties}
We prove the assertion on the growth of the sample paths of $B_{X_t}$.

\begin{proof}[Proof of Theorem \ref{thm_short_long_time}]
This is a simple application of the results in \cite{schilling1998growth}. Denote by $\sigma(\xi)=f(|\xi|^2)$ the symbol of the process $B_{X_t}$. As in Example 5.5 of \cite{schilling1998growth} we compute the Blumenthal-Getoor-type indices
\begin{align*}
\beta_\infty & = \inf \left\{r>0 \middle| \lim_{|\xi| \to \infty} \tfrac{\sigma(\xi)}{|\xi|^r} = 0\right\} \\
\delta_\infty & = \inf \left\{r>0 \middle| \liminf_{|\xi| \to \infty} \tfrac{\sigma(\xi)}{|\xi|^r} = 0\right\}. 
\end{align*}
By Proposition \ref{prop_growth} (\ref{ass_expansion}), $f$ has the asymptotic expansion  
\begin{equation} \label{eq_asymptotic_in_x}
f(\lambda) \sim \overline{m}(0, \infty) - \sum_{k=0}^\infty \Gamma(-\alpha+k) p_k \lambda^{\alpha-k}
\end{equation}
as $\lambda \to \infty$. Thus $\beta_\infty=\delta_\infty=2 \alpha$ and the claim follows from Theorem 4.6 of \cite{schilling1998growth}. 
\end{proof}

We also give a proof of Theorem \ref{thm_probab_interpretation} yielding a probabilistic interpretation of the coefficient $a_0$ in the asymptotic expansion of $m$ under Hypothesis \ref{hypo_main_hypothesis}.  This interpretation will hinge on L\'evy's arcsine law. To this end recall the notion of regularly varying functions.

\begin{definition}[\cite{bertoin1998levy}, Chapter 0.7]
A measurable function $f: (0, \infty) \to (0, \infty)$ is \emph{regularly varying} at $\infty$ if for every $\tau>0$ the ratio $f(\tau \lambda)/f(\lambda)$ converges in $(0, \infty)$ as $\lambda \to \infty$. Indeed, there is a real number $\rho$, called the \emph{index} such that
\[
\lim_{\lambda \to \infty} \tfrac{f(\tau \lambda)}{f(\lambda)} = \tau^\rho
\]
for every $\tau>0$.
\end{definition}

\begin{proof}[Proof of Theorem \ref{thm_probab_interpretation}]
By equation \eqref{eq_asymptotic_in_x} we find for $\tau>0$ that
\begin{equation} \label{eq_arcsine_limit}
\tfrac{f(\tau \lambda)}{f(\lambda)} \to \tau^{\alpha}
\end{equation}
as $\lambda \to \infty$ since in \eqref{eq_asymptotic_in_x} the highest power in $\lambda$ is of the order $\alpha>0$. This means that $f$ is regularly varying at $\infty$ with index $\alpha$.

By L\'evy's arcsine law (Theorem III.3.6 of \cite{bertoin1998levy}) we have
\begin{equation} \label{eq_probab_for_a0}
\lim_{x \to 0^+} \tfrac{1}{x} \mathbbm E\left(X_{T(x)^-}\right) = \alpha,
\end{equation}
where $T(x)=\inf\left\{ t \geq 0 | X_t > x \right\}$ is the first passage time strictly above $x$.

The assertion that the coefficient $p_0$ can be described probabilistically is clear since $\lambda^{-\alpha}f(\lambda) \to - \Gamma(-\alpha) p_0$ for $\lambda \to \infty$. As both $a_0$ and $f$ can be described probabilistically in terms of $X_t$ via \eqref{eq_probab_for_a0} and $\mathbbm E\left(e^{-\lambda X_t} \right) = e^{- t f(\lambda)}$, respectively, so can $p_0$:
\[
- \Gamma(-\alpha) p_0 = \lim_{\lambda \to \infty} \lambda^{-\alpha} f(\lambda) = \lim_{\lambda \to \infty} - \lambda^{-\alpha} \tfrac{1}{t} \log \mathbbm E\left(e^{- \lambda X_t} \right)
\]
for $t>0$. Likewise for the other coefficients $p_k$. 
\end{proof}

%------------------------------------------------------------------------------------------------------------------------
\subsection{Constructing a classical pseudodifferential operator}
We shift the the generator $A$ by a suitable constant to obtain a classical pseudodifferential operator $\tilde A = A - \overline{m}(0, \infty) I$. The classicality is crucial in order to define the regularized zeta function and generalized heat trace. The heat kernel will be a classical pseudodifferential operator if $\tilde A$ is invertible and parameter elliptic with respect to a sector in the complex plane, cf. Definition \ref{def_elliptic_parameter_elliptic} (\ref{def_parameter_elliptic}). 

\subsubsection{Classicality of the symbol}
It is easy to see that a suitably adjusted version of $f(|\xi|^2)$ is a classical symbol.

\begin{proposition} \label{prop_classical}
Assume Hypothesis \ref{hypo_main_hypothesis} and define a function
\[
\tilde{\sigma}(\xi)=f(|\xi|^2) -\overline{m}(0, \infty).
\]
Then $\tilde{\sigma}$ is a symbol in $S^{2 \alpha, 0}(\R^n)$. Moreover, it belongs to $S^{2 \alpha, 0}_{cl}(\R^n)$, i.e. it is classical with asymptotic expansion 
\begin{equation} \label{eq_sigma_classical_expansion}
\tilde{\sigma}(\xi) \sim - \sum_{k=0}^\infty \Gamma(-\alpha+k) p_k  |\xi|^{2(\alpha-k)}
\end{equation}
in the sense of Definition \ref{def_classical_expansion}. 
\end{proposition}

\begin{proof}
1. We first show that $\tilde{\sigma} \in S^{2 \alpha, 0}(\R^n)$. By Definition \ref{def_symbol} it suffices to show that $f(|\xi|^2)$ is a smooth function on $\R^n$ and that for all multi-indices $\beta \in \N^n_0$ there is a constant $C_\beta$ with
\begin{equation} \label{eq_symbol}
\left| \partial_\xi^\beta f(|\xi|^2) \right| \leq C_\beta \langle \xi \rangle^{2\alpha-|\beta|} 
\end{equation}
for any $\xi \in \Rn$. Recall that $\langle \xi \rangle = \sqrt{1+|\xi|^2}$. 

Now, $f(|\xi|^2)$ is smooth as $f$ is smooth on $[0, \infty)$ by Proposition \ref{prop_growth} (\ref{ass_smoothness}). 

To obtain the estimates in \eqref{eq_symbol} recall Faa di Bruno's formula for the chain rule in the version of equation (6.8) of \cite{shubin2001pseudodifferential}. Let $\beta \in \N^n$ be a multi-index. Then
\[
\partial_\xi^\beta f(|\xi|^2)=\sum_{\gamma_1 + \cdots + \gamma_p=\beta} c_{p, \gamma_1, \ldots, \gamma_p} \left(\partial_\lambda^p f\right)(|\xi|^2) \partial_\xi^{\gamma_1} |\xi|^2 \cdots \partial_\xi^{\gamma_p} |\xi|^2,
\]
where the sum runs over all partitions of the multi-index $\beta$ into sums of nonzero multi-indices $\gamma_1, \ldots, \gamma_p$ for $p=1, 2, \ldots$ As $|\xi|^2$ is a symbol in $S^{2,0}(\R^n)$ we have 
\[
\left|\partial_\xi^{\gamma_i} |\xi|^2 \right| \leq C_{\gamma_i} \langle \xi \rangle^{2-|\gamma_i|}
\]
for some $C_{\gamma_i}$. Also, $|\gamma_1| + \cdots + |\gamma_p|=|\beta|$ by construction so that overall
\[
\left| \partial_\xi^{\gamma_1} |\xi|^2 \cdots \partial_\xi^{\gamma_p} |\xi|^2 \right| \leq C_{\gamma_1, \ldots, \gamma_p} \langle \xi \rangle^{2p-|\beta|}
\]
for a constant $C_{\gamma_1, \ldots, \gamma_p}$. From Proposition \ref{prop_growth} (\ref{ass_bound_for_ellipticity}) we find
\[
\left|\left(\partial_\lambda^p f\right)(|\xi|^2)\right| \leq C_p |\xi|^{2 (\alpha-p)} \leq C_p \langle \xi \rangle^{2 (\alpha-p)}
\]
for some constants $C_p$ so that overall \eqref{eq_symbol} follows.

2. Now prove the asymptotic expansion \eqref{eq_sigma_classical_expansion}. Let $\chi \in C^\infty(\R^n)$ be a cutoff function which is equal to 1 outside $|\xi| \geq 1$ and equal to zero for $|\xi| \leq 1/2$. We must show
\begin{equation*} 
\sigma(\xi) - \overline{m}(0, \infty) + \sum_{k=0}^{N-1} \chi(\xi) \Gamma(-\alpha+k) p_k  |\xi|^{2(\alpha-k)} \in S^{2 \alpha-N, 0}(\R^n)
\end{equation*}
for any $N \geq 1$.

Suppose $|\xi| \geq 1$ and consider $f(|\xi|^2) - \overline{m}(0, \infty) + \sum_{k=0}^{N-1} \Gamma(-\alpha+k) p_k  |\xi|^{2(\alpha-k)}$. We can rewrite this as $f_N(|\xi|^2)$ with
\[
f_N(\lambda)=\int_0^\infty \left(1-e^{-\lambda t}\right)\left(m(t)-\sum_{k=0}^{N-1} p_k t^{-1-\alpha+k}\right) dt.
\]
Now repeat the arguments of Proposition \ref{prop_growth} (\ref{ass_bound_for_ellipticity}) to see that
\[
\left| f_N^{(l)}(\lambda) \right| \leq C_{l, N} \lambda^{\alpha-N-l}
\]
for some constants $C_{l, N}$. An argument as in 1. shows that $f_N$ belongs to $S^{2\alpha-N,0}(\R^n)$ proving the claim. 
\end{proof}

\subsubsection{Ellipticity of the symbol}
We note the (parameter-)ellipticity of a shifted version of the Bernstein function.

\begin{proposition} \label{prop_elliptic}
Assume Hypothesis \ref{hypo_main_hypothesis} and set $\tilde{\sigma}(\xi)=f(|\xi|^2) -\overline{m}(0, \infty)$.
Then $\tilde{\sigma}$ is elliptic. Moreover, for $\theta \in (\pi/4, \pi/2)$ the symbol $\tilde\sigma$ is $\Lambda$-elliptic.
\end{proposition}

\begin{proof}
1. For the ellipticity we must show that there is a $C>0$ with 
\begin{equation} \label{eq_elliptic}
C \langle \xi \rangle^{2\alpha} \leq \tilde{\sigma}(\xi)
\end{equation}
for any $\xi \in \Rn$. 

By Proposition \ref{prop_growth} (\ref{ass_bound_for_ellipticity}) there is an $R>0$ such that for $|\xi|^2 > R$ we have $C_1 |\xi|^{2\alpha} < \sigma(\xi)$ where $C_1 = - \tfrac{1}{2} \Gamma(-\alpha) p_0$. Without loss of generality $R>1$ so that for $|\xi|^2 > R$ we have $C_2 \langle \xi \rangle^{2 \alpha} \leq \sigma(\xi)$ with $C_2=(\tfrac{1}{2})^\alpha C_1$ as $\tfrac{1}{2} \langle \xi \rangle^2 \leq |\xi|^2$ for $|\xi| \geq 1$. There is an $R' \geq R$ such that
\begin{align*}
\tfrac{1}{2} C_2 \langle \xi \rangle^{2\alpha} & \leq C_2 \langle \xi \rangle^{2\alpha} - \overline{m}(0,\infty) \\
& \leq f(|\xi|^2) - \overline{m}(0, \infty) \\
& = \tilde{\sigma}(\xi)
\end{align*}
for $|\xi|^2 \geq R'$.  

If $|\xi|^2 \leq R'$ we have
\[
-\overline{m}(0, \infty) \langle R' \rangle ^{-2\alpha} \langle \xi \rangle^{2\alpha} \leq -\overline{m}(0, \infty) \leq \tilde{\sigma}(\xi).
\]
Thus, \eqref{eq_elliptic} holds with $C=\min\{\tfrac{1}{2}C_2; -\overline{m}(0, \infty) \langle R' \rangle^{-2\alpha}\}$.  

2. Now for the parameter-ellipticity. We must show that 
\[
|(\lambda-\sigma(\xi))^{-1}| \leq C \langle \xi \rangle ^{-2\alpha}
\]
for some $C>0$ and all $\xi \in \Rn$. Equivalently $C^{-1} \langle \xi \rangle^{4 \alpha} \leq |\lambda - \sigma(\xi)|^2$. Let $\lambda \in \Lambda$ with $\lambda=\lambda_1 + i \lambda_2$ and $\lambda_i \in \R$. We distinguish cases according to the sign of $\lambda_1$.

Case $\lambda_1 \leq 0$: we find
\begin{align*}
|\lambda-\tilde\sigma(x,\xi)|^2 & = (\lambda_1-\sigma(\xi))^2+\lambda_2^2 \\
& \geq \tilde\sigma(\xi)^2 \\
& \geq C' \langle \xi \rangle^{4 \alpha}
\end{align*}
for some constant $C'$ by ellipticity and since $\lambda_1$ and $-\sigma$ have the same sign.

Case $\lambda_1 > 0$: here $\lambda_2 \geq \lambda_1$ since we assumed $\theta>\pi/4$. Thus,
\begin{align*}
|\lambda-\tilde\sigma(\xi)|^2 & = (\lambda_1-\sigma(\xi))^2+\lambda_2^2 \\
& \geq (\lambda_1 - \tilde\sigma(\xi))^2 + \lambda_1^2 \\
& \geq \tfrac{1}{2} \tilde\sigma(\xi)^2 \\
& \geq \tfrac{1}{2} C' \langle \xi \rangle^{4 \alpha}
\end{align*}
for some $C'$ by ellipticity. Here we use the fact that the function $g(x)=(x-c)^2+x^2$ for $c>0$ has a minimum located at $x=c/2$ with value $g(c/2)=c^2/2$. 
\end{proof}

\subsubsection{The spectrum of the operator}
We collect useful facts about the spectrum of the operator $\tilde A=A-\overline{m}(0, \infty)$.

\begin{proposition} \label{prop_spectrum}
Assume Hypothesis \ref{hypo_main_hypothesis} and set $\tilde{\sigma}(\xi)=f(|\xi|^2) -\overline{m}(0, \infty)$. Let $\tilde A$ be the pseudodifferential operator with symbol $\tilde\sigma$.
\begin{enumerate}[(i)] 
\item $\tilde{A}$ is essentially selfadjoint, i.e. has self-adjoint closure $\overline{\tilde{A}}$.
\item The spectrum of $\overline{\tilde{A}}$ is contained in the positive real line so that in particular $\tilde{A}$ is invertible: $\Sp\left(\overline{\tilde{A}}\right) \subset \left[-\overline{m}(0, \infty), \infty\right)$. \label{spectrum}
\end{enumerate}
\end{proposition}

Note that the closure of $\tilde{A}$ generates an analytic semigroup which by abuse of notation we denote by $e^{-t\tilde{A}}$. The heat operator is a classical pseudodifferential operator. 

\begin{corollary}
Under Hypothesis \ref{hypo_main_hypothesis} the operator $-\tilde{A}$ extends to the generator of an analytic semigroup $e^{-t\tilde{A}}$ on $L^2(\Rn)$. The heat operator $e^{-t\tilde{A}}$ is a classical pseudodifferential operator with symbol $\sigma(e^{-t \tilde{A}})(\xi)=e^{-t \tilde{\sigma}(|\xi|^2)}$ belonging to $S^{-\infty, 0}_{cl}(\Rn)$.
\end{corollary}

\begin{proof}
The spectrum of $\overline{\tilde{A}}$ is entirely contained in the interval $[-\overline{m}(0, \infty), \infty)$ by Proposition \ref{prop_spectrum} (\ref{spectrum}). Thus, the spectrum of $-\overline{\tilde{A}}$ is contained in a sector in the left half of the complex plane and the analytic semigroup $e^{-t \tilde{A}}$ exists by Theorem II.4.6 of \cite{engel2000one}.

The fact that $e^{-t \tilde{A}}$ is a pseudodifferential operator follows from a simple application of Theorem 4.1 of \cite{maniccia2013determinants} whose hypotheses are satisfied: the symbol $\tilde{\sigma}$ is of positive order in $\xi$ and order $0$ in $x$. Moreover, it is $\Lambda$-elliptic by Proposition \ref{prop_elliptic}; and the resolvent of $\overline{\tilde{A}}$ exists in the whole sector $\Lambda$ as $0 \not \in \Sp(\tilde{A})$ by Proposition \ref{prop_spectrum} (\ref{spectrum}). 
\end{proof}

\begin{proof}[Proof of Proposition \ref{prop_spectrum}]
(i) By standard arguments (cf. for example Chapter 1 of \cite{schmudgen2012unbounded} or Chapter 2.1 of \cite{boggiatto1996global}), $\tilde A$ is closable. Also, $\tilde A$ is symmetric, i.e. $(\tilde A u, v)=(u, \tilde A v)$ in $L^2(\R^n)$ for any $u, v \in \mathcal S(\R^n)$ since $\tilde \sigma$ is real-valued and independent of $x$. By Proposition 1.3 of \cite{boggiatto1996global}, $\overline{\tilde A}$ is selfadjoint if and only if $\tilde A$ is symmetric. 

(ii) Spectrum of $\tilde{A}$: clearly, $\Sp\left(\overline{\tilde{A}}\right) \subseteq [0, \infty)$. Proposition 3.10 (ii) of \cite{schmudgen2012unbounded} implies that for a selfadjoint densely defined operator $T$ we have $\lambda \not\in \Sp(T)$ if and only if $||(T-\lambda)u|| \geq c_\lambda ||u||$ for some constant $c_\lambda>0$ and all $u$ in the domain of $T$. 
In our case note that for $u \in \mathcal S(\Rn)$ we have $||\tilde{A} u||^2 \geq (\overline{m}(0,\infty))^2 ||u||^2$ by Plancherel's theorem as the symbol of $\tilde{A}$ is bounded below by $-\overline{m}(0,\infty)$ so the claim follows.

\end{proof}

%------------------------------------------------------------------------------------------------------------------------
\subsection{Zeta function and heat trace of the shifted generator}
For a classical pseudodifferential operator we compute the regularized zeta-function and the generalized heat trace. 

\begin{proof}[Proof of Theorem \ref{thm_generator_as_pseudo}]
(i) The assertions follow from Propositions \ref{prop_classical} and \ref{prop_elliptic}.

(ii) We set $\alpha_k = -p_k \Gamma(-\alpha+k)$ for $k=0,1,2, \ldots$ for short and apply Theorem 4.1 of \cite{maniccia2013determinants}. Let $\Lambda$ be a sector as in Definition \ref{def_elliptic_parameter_elliptic} (\ref{def_parameter_elliptic}) with $\theta$ as in Proposition \ref{prop_elliptic}.

To find the symbol of the heat operator we first construct the symbol expansion of the parameter-dependent parametrix using the algorithm of Proposition 3.2 of \cite{maniccia2006complex} for symbols independent of $x$. Recall that by a \emph{parameter-dependent parametrix} of a pseudodifferential operator $\tilde A$ with symbol $\tilde \sigma(\xi) \in S^{r,0}$ we mean an inverse modulo smoothing operators, i.e. operators with symbol in $S^{-\infty, 0}(\R^n)$. If we denote the symbol of the parametrix by $b(\xi, \lambda)$ then on the level of symbols we require that $(\lambda-\tilde\sigma(\xi)) b(\xi, \lambda)-1 \in S^{-N, 0}(\R^n)$ for any $N \in \N$. The ansatz $b(\xi, \lambda) \sim b_{-r}(\xi, \lambda) + b_{-r-1}(\xi, \lambda)+ b_{-r-2}(\xi, \lambda) + \cdots$ leads to the relation of formal power series
\[
\left[(\lambda-\tilde\sigma_r(\xi)) + \tilde\sigma_{r-1}(\xi) +\cdots \right] \left[b_{-r}(\xi, \lambda) + b_{-r-1}(\xi, \lambda)+ \cdots \right] = 1.
\]
Roughly speaking one collects terms according to the degree of homogeneity in $\xi$ where $\lambda-\tilde\sigma_r$ is treated as order $r$.

In our case with $r=2\alpha$ the first terms are given as
\begin{equation} \label{eq_explicit_parametrix}
\left. \begin{aligned}
b_{-2 \alpha}(\xi, \lambda) & = \tfrac{1}{\lambda - \alpha_0 |\xi|^{2 \alpha}} \\
b_{-2 \alpha-1}(\xi, \lambda) & = 0 \\
b_{-2 \alpha-2}(\xi, \lambda) & = \tfrac{\alpha_1 |\xi|^{2 \alpha-2}}{(\lambda - \alpha_0 |\xi|^{2 \alpha})^2} \\
b_{-2 \alpha-3}(\xi, \lambda) & = 0 \\
b_{-2 \alpha-4}(\xi, \lambda) & = \tfrac{\alpha_2 |\xi|^{2 \alpha-4}}{(\lambda-\alpha_0 |\xi|^{2 \alpha})^2} + \tfrac{\alpha_1^2 |\xi|^{4 \alpha-4}}{(\lambda-\alpha_0 |\xi|^{2 \alpha})^3}
\end{aligned} \right\rbrace
\end{equation}
The terms of order $-2 \alpha-k$ with $k$ odd all vanish because the asymptotic expansion of $\tilde\sigma$ contains no terms of order $2\alpha-k$ with $k$ odd. 

We can then determine the symbol expansion for the heat operator using the above parametrix to obtain
\[
\sigma(e^{-t\tilde{A}})(\xi) \sim \sum_{k=0}^\infty \frac{1}{2 \pi i} \int_{\partial \Lambda} e^{-t \lambda} b_{-2 \alpha-k}(\xi, \lambda) d\lambda
\]
by Theorem 4.1 of \cite{maniccia2013determinants}, where $\partial \Lambda$ is a parametrization of the boundary of the sector $\Lambda$. To the above lowest order terms, the symbol is thus given as
\begin{align} \label{eq_explicit_heat_expansion}
\sigma(e^{t\tilde{A}})(\xi) \sim & e^{-t \alpha_0 |\xi|^{2 \alpha}} + \left[-\alpha_1 |\xi|^{2 \alpha-2}-\alpha_2 |\xi|^{2 \alpha-4} \right] t e^{-t \alpha_0 |\xi|^{2 \alpha}} \nonumber \\
& + \tfrac{1}{2} \alpha_1^2 |\xi|^{4 \alpha-4} t^2 e^{-t \alpha_0 |\xi|^{2 \alpha}} + \cdots
\end{align}
which follows by integration by parts (detailed computations in Section \ref{sec_example}). 
\end{proof}

We can perform a similar computation for the $\zeta$-function of the operator. 

\begin{proof}[Proof of Theorem \ref{thm_zeta_noninteger_order}]
Let $\Lambda$ be a sector as in Definition \ref{def_elliptic_parameter_elliptic} (\ref{def_parameter_elliptic}) with $\theta$ as in Proposition \ref{prop_elliptic}. For the construction of the complex powers of $\tilde A$ with classical symbol $\sigma$ we must have that
\begin{enumerate}[(i)]
\item $\lambda-\tilde A$ is invertible for all $0 \neq \lambda \in \Lambda$ and
\item $\lambda=0$ is at most an isolated spectral point.
\end{enumerate}
These assumptions are satisfied in our case by Proposition \ref{prop_spectrum} (\ref{spectrum}). Then define the complex powers $\tilde A^{-z}$ for $z \in \C, \re z > 0$ by a Dunford integral
\[
\tilde A^{-z}=\tfrac{1}{2 \pi i} \int_{\partial \Lambda_\epsilon} \lambda^{-z} (\lambda-\tilde A)^{-1} d\lambda
\]
for $\re z > 0$ where $\Lambda_\epsilon=\Lambda \cup \{|\lambda| \leq \epsilon\}$ and $\partial \Lambda_\epsilon$ is a parametrization of the boundary of $\Lambda_\epsilon$ with the circular part traversed clockwise. We refer to \cite{maniccia2006complex}, Section 3.2 to see that $\tilde A^{-z}$ is a pseudodifferential operator in $S^{2\alpha z, 0}(\Rn)$ with symbol $\sigma(\xi; z)$ and
\[
\sigma(\xi; z)=\tfrac{1}{2 \pi i} \int_{\partial \Omega_{\langle \xi \rangle}} \lambda^{-z} (\lambda-\sigma(\xi))^{-1} d \lambda.
\]
Here, $\Omega_{\langle \xi \rangle} = \left\{z \in \C \setminus \Lambda \middle | \tfrac{1}{c} \langle \xi \rangle^r < |z| < c \langle \xi \rangle^r \right\}$ for a suitable constant $c \geq 1$ such that $\sigma(\xi) \in \Omega_{\langle \xi \rangle}$ for any $\xi \in \R^n$, cf. Lemma 3.2 of \cite{maniccia2006complex}. The symbol is defined modulo smoothing symbols in $S^{-\infty, 0}(\Rn)$ depending analytically on $z$. 

The asymptotic symbol expansion of $\tilde A^{-z}$ is given for $\re z>0$ as
\[
\sigma\left(\tilde A^{-z}\right)(\xi) \sim  \sigma_{-2 \alpha z}(\xi; z) + \sigma_{-2 \alpha z-2}(\xi; z) + \sigma_{-2 \alpha z-4}(\xi; z) + \cdots
\]
with
\begin{equation} \label{eq_explicit_zeta_expansion}
\left. \begin{aligned}
\sigma_{-2 \alpha z}(\xi; z) & = \alpha_0^{-z} |\xi|^{-2 \alpha z} \\
\sigma_{-2 \alpha z-2}(\xi; z) & = -\alpha_0^{-z-1} \alpha_1 z |\xi|^{-2 \alpha z-2} \\
\sigma_{-2 \alpha z-4}(\xi; z) & = \left[-\alpha_0^{-z-1} \alpha_2 z  + \tfrac{1}{2}\alpha_0^{-z-2} \alpha_1^2 z (z+1)\right] |\xi|^{-2 \alpha z-4}
\end{aligned} \right\rbrace
\end{equation}
obtained by integration by parts (exemplary computations in Section \ref{sec_example}). 

Now for the regularized zeta function $\zeta(z)=TR(\tilde A^{-z})$. By Theorem 2.7 of \cite{maniccia2013determinants}, this function is analytic for $\re z > n/2\alpha$ and can be meromorphically extended to the entire complex plane with at most simple poles at the points $z_k=(n-k)/2\alpha$ for $k=0, 1, 2, \ldots$. The behaviour of $\zeta(z)$ near $z_k$ is described as
\begin{equation} \label{eq_defn_TR}
\tfrac{(2\pi)^{-n}}{2\alpha (z-z_k)} \tfrac{1}{n} \int_{|\xi|=1} \sigma_{-2\alpha z + 2\alpha z_k -n}(\xi; z) d \omega(\xi)
\end{equation}
with $\omega$ the surface measure on the unit sphere in $\R^n$. Compared to (2.10) in \cite{maniccia2013determinants} all terms with averaging in the $x$-direction are omitted, the finite-part integral vanishes anyway on operators with symbols independent of $x$.
 
\end{proof}

Based on the singularity structure of the regularized zeta-function we can give the short-term asymptotics of the generalized heat trace.

\begin{proof}[Proof of Theorem \ref{thm_heat_trace_noninteger_order}]
This is an application of Theorem 4.3 in \cite{maniccia2013determinants}. It relies on a mathematical folklore written up in Section 5 of \cite{grubb1996zeta}. The singularity structure of the generalized heat trace is related via the Mellin transform to $\Gamma(z) \zeta(z)$ so that we need to consider the pole structure of $\Gamma(z) \zeta(z)$. Indeed, for $2 \alpha$ rational
\begin{equation} \label{eq_gamma_zeta}
\Gamma(z) \zeta(z) \sim \sum_{k =0}^\infty \tfrac{c_k}{z-(n-k)/2 \alpha} + \sum_{l=1}^\infty \tfrac{\tilde{c}_l}{[z-(-l)]^2}
\end{equation}
with coefficients $c_k$ and $\tilde{c}_l$ as in the assertion. We have double poles where the poles of $\Gamma(z)$ and $\zeta(z)$ coincide which can only occur for $(n-k)/2 \alpha$ a negative integer denoted by $l$ in equation \eqref{eq_gamma_zeta}.

For $2 \alpha$ irrational, $\Gamma(z) \zeta(z)$ has only simple poles
$\Gamma(z) \zeta(z) \sim \sum_{k=0}^\infty \tfrac{c_k}{z-(n-k)/2 \alpha}$
with $c_k$ as given. 
\end{proof}

%%%%%%%%%%%%%%%%%%%%%%%%%%%%%%%%%%%%%%%%%%%%%%%%%%%%%%%%%%%%%%%%%%%%%%%%%%%%%%%
\section{Worked example: the relativistic stable L\'evy process} \label{sec_example}
To illustrate the approach and to show that it is analytically tractable we consider the Bernstein function $f(\lambda)=\sqrt{\lambda+1}-1$ in dimensions $n=2$ and $n=3$. The generator of this process is usually denoted by $-H$ where $H=(-\Delta+1)^{1/2}-1$ is the relativistic Hamiltonian with zero potential and mass 1 and $\Delta$ is the Laplace operator. The behaviour of the regularized zeta function and the generalized heat trace is different in these dimensions with the appearance of logarithmic terms in dimension 3.

\subsection{Dimension-independent considerations}
We first derive several properties of the Bernstein function and the associated pseudodifferential operator that are independent of the dimension $n$. \\

\textbf{Asymptotic expansion of the Bernstein function.} The L\'evy density of the Bernstein function is given by $m(t)=\tfrac{1}{2 \sqrt{\pi}} t^{-3/2} e^{-t}$ as $\Gamma(1/2)=\sqrt{\pi}$. From the Taylor series for $e^{-t}$ near $t=0$ we find
\[
m(t) \sim \tfrac{1}{2 \sqrt{\pi}} t^{-3/2} \left(1-t+ \tfrac{1}{2} t^2-\tfrac{1}{3!} t^3 \pm \cdots \right).
\]
In our previous notation this yields $\alpha = -\tfrac{1}{2}$ and $p_k = \tfrac{(-1)^k}{k!} \tfrac{1}{2 \sqrt{\pi}}$ for $k=0, 1, 2, \ldots$

We apply Proposition \ref{prop_growth} (\ref{ass_expansion}) to find the asymptotics of $f$. With 
$\overline{m}(t)=m(t)-\tfrac{1}{2 \sqrt{\pi}} t^{-3/2}$
we find $\overline{m}(0, \infty) = -1$. 
Thus, 
\[
f(\lambda) \sim \lambda^{1/2} -1 + \tfrac{1}{2} \lambda^{-1/2} - \tfrac{1}{8} \lambda^{-3/2} + \tfrac{1}{16} \lambda^{-5/2} \mp \cdots
\]
which agrees with the Taylor expansion of $f$ in $1/\lambda$. \\

\textbf{The associated pseudodifferential operator.} The associated pseudodifferential operator $A$ has the symbol $\sigma(\xi)=(|\xi|^2+1)^{1/2}-1$. We consider the shifted operator $\tilde{A}=A+I$ with symbol
\[
\tilde\sigma(\xi)=f(|\xi|^2)+1 \sim |\xi| + \tfrac{1}{2} |\xi|^{-1} - \tfrac{1}{8} |\xi|^{-3} + \tfrac{1}{16} |\xi|^{-5} \mp \cdots
\]
which is a classical symbol in $S_{cl}^{1,0}(\R^2)$. 

Let $\Lambda$ be the sector $\Lambda=\left\{z \in \C \middle| \theta \leq \arg(z) \leq 2\pi-\theta\right\}$ for a $\theta \in (\pi/4, \pi/2)$. A quick calculation shows that the symbol is $\Lambda$-elliptic, i.e. that $\tilde\sigma(\xi)$ does not take values in $\Lambda$ and $\tfrac{1}{\sqrt{2}} \langle \xi \rangle \leq |\lambda-\tilde\sigma(\xi)|$ for all $\lambda \in \Lambda$, $\xi \in \R^2$. \\

\textbf{Parameter-dependent parametrix.} For the parameter-dependent parametrix we apply the algorithm of the proof of Theorem \ref{thm_generator_as_pseudo} and find
\begin{align*}
b_{-1}(\xi, \lambda) & = \frac{1}{\lambda- |\xi|}\\
%b_{-2}(\xi, \lambda) & = 0\\
b_{-3}(\xi, \lambda) & = \frac{\tfrac{1}{2} |\xi|^{-1}}{(\lambda-|\xi|)^2}\\
%b_{-4}(\xi, \lambda) & = 0\\
b_{-5}(\xi, \lambda) & = \frac{\left(\tfrac{1}{2} |\xi|^{-1} \right)^2}{(\lambda-|\xi|)^3}-\frac{\tfrac{1}{8} |\xi|^{-3}}{(\lambda-|\xi|)^2},
\end{align*}
which agrees with \eqref{eq_explicit_parametrix}.
Putting everything together we arrive at 
\[
b(\xi, \lambda) \sim \frac{1}{\lambda- |\xi|} + \frac{\tfrac{1}{2} |\xi|^{-1}}{(\lambda-|\xi|)^2} + \left[\frac{\left(\tfrac{1}{2} |\xi|^{-1} \right)^2}{(\lambda-|\xi|)^3}-\frac{\tfrac{1}{8} |\xi|^{-3}}{(\lambda-|\xi|)^2} \right] + \cdots
\]
for the parameter-dependent parametrix. \\

\textbf{Symbol expansion of the heat operator.} The first few terms in the asymptotic expansion of the symbol of the heat operator are given by 
\[
\sigma\left(e^{-t\tilde{A}}\right)(\xi) \sim e^{-t |\xi|} + \left[-\tfrac{1}{2} |\xi|^{-1} + \tfrac{1}{8} |\xi|^{-3} \right] t e^{-t |\xi|} + \tfrac{1}{8} |\xi|^{-2}t^2 e^{-t |\xi|} + \cdots
\]
in the lowest orders. \\

\textbf{The regularized zeta function.} The lowest-order terms in the asymptotic expansion of the complex powers is given by
\begin{equation} \label{eq_as_exp_zeta_example}
\sigma\left(\tilde{A}^{-z}\right)(\xi) \sim |\xi|^{-z} -\tfrac{1}{2} z |\xi|^{-z-2} + \tfrac{1}{8}(z^2+2z) |\xi|^{-z-4} + \cdots 
\end{equation}
for the complex powers.

\subsection{Dimension $n=2$}
The regularized zeta function has at most simple poles at the points $z_k=2-k$ for $k=0, 1, 2, \ldots$ and the residues are given in terms of the homogeneous symbol of order $-n=-2$. Let $\omega$ be the surface measure on the unit sphere in $\R^2$. 

\begin{enumerate}
\item[$k=0$:] near $z_0=2-0=2$ we find that
\[
\zeta(z) = \tfrac{(2\pi)^{-2}}{-2+z} \tfrac{1}{2} \int_{|\xi|=1} |\xi|^{-2} d \omega(\xi) = \tfrac{1}{4 \pi} \cdot \tfrac{1}{z-2}
\]
\item[$k=1$:] near $z_1=2-1=1$ we see that $\zeta(z)$ is analytic as the asymptotic expansion \eqref{eq_as_exp_zeta_example} contains no homogeneous term of order $-2$.
\item[$k=2$:] near $z_2=2-2=0$, the homogeneous term of order $2$ is given by $-\tfrac{1}{2} z |\xi|^{-z-2}$ which vanishes at $z_2=0$ so that $\zeta(z)$ is analytic there.
\item[$k=3$:] near $z_3=2-3=-1$ we find that $\zeta(z)$ is analytic as the asymptotic expansion \eqref{eq_as_exp_zeta_example} contains no homogeneous term of order $-2$.
\item[$k=4$:] near $z_4=2-4=-2$ we have
\[
\zeta(z) = \tfrac{(2\pi)^{-2}}{2+z} \tfrac{1}{2} \int_{|\xi|=1} \tfrac{1}{8}\left[z^2+2 z\right]_{z=-2} |\xi|^{-2} d \omega(\xi) = 0,
\]
so that $\zeta(z)$ is analytic there.
\end{enumerate}
Hence to lowest orders the only pole of $\zeta(z)$ is at $z=2$ with residue $\tfrac{1}{4 \pi}$. \\

\textbf{Heat trace expansion.} We finally compute the coefficients of the heat trace expansion. The ansatz is to write 
\[
\Gamma(z) \zeta(z) = \sum_{k=0}^\infty \left[\tfrac{c_k}{z-(2-k)} + \tfrac{\tilde{c}_k}{\left(z-(2-k)\right)^2}\right] + g(z)
\]
where $g$ is an entire function which is of no concern to us. The $\Gamma$-function has only simple poles located at $z_k=-k$ for $k=0, 1, 2, \ldots$ with residue $\res_{z=-k} \Gamma(z)=\tfrac{(-1)^k}{k!}$. An explicit computation yields the following:
\begin{align*}
c_0 & = \res_{z=2} \Gamma(z) \zeta(z) = \Gamma(2) \cdot \tfrac{1}{4 \pi}=\tfrac{1}{4 \pi} \\
c_1 & = 0 \\
c_2 & =1 \cdot \zeta(0) \\
c_3 & = \res_{z=-1} \Gamma(z) \zeta(z) = \tfrac{(-1)^1}{1!} \cdot \zeta(-1) = - \zeta(-1) \\
c_4 & = \res_{z=-2} \Gamma(z) \zeta(z) = \tfrac{(-1)^2}{2!} \cdot \zeta(-2)=\tfrac{1}{2} \zeta(-2)
\end{align*}
with $\tilde{c}_0=\tilde{c}_1=\tilde{c}_2=\tilde{c}_3=\tilde{c}_4=0$ as there are no double poles of $\Gamma(z) \zeta(z)$ at these points.

Hence the desired heat trace expansion for $\tilde A$ reads
\[
TR\left(e^{-t \tilde A} \right) \sim \tfrac{1}{4 \pi} t^{-2} + \zeta(0) - \zeta(-1) t^{1} +\tfrac{1}{2} \zeta(-2) t^{2}  + \cdots
\]
in agreement with \eqref{eq_lowest_coeffs_heat_trace}. Note that the heat trace expansion of the relativistic L\'evy process $TR\left(e^{-t A} \right)$ can be obtained via the relation $TR\left(e^{-t A} \right)=e^{t} TR\left(e^{-t \tilde A} \right)$ as $\tilde A=A+I$ so that
\begin{equation*}
TR\left(e^{-t A} \right) \sim \tfrac{1}{4 \pi} e^t t^{-2} + \zeta(0) e^t - \zeta(-1) e^t t^{1} +\tfrac{1}{2} \zeta(-2) e^t t^{2}  + \cdots.
\end{equation*}

In dimension $n=2$, the zeta-function is regular at $-k$ for $k=0, 1, 2, \ldots$ (and so there are no logarithmic terms in the heat trace expansion) as $n=1/\alpha$, i.e. there is a algebraic relation between the dimension $n$ and the parameter $\alpha$. 

\subsection{Dimension $n=3$}
The regularized zeta function has at most simple poles at the points $z_k=3-k$ for $k=0, 1, 2, \ldots$ with the following residues. Let $\omega$ be the surface measure on the unit sphere in $\R^3$.

\begin{enumerate}
\item[$k=0$:] near $z_0=3-0=3$ we find that
\[
\zeta(z) = \tfrac{(2\pi)^{-3}}{-3+z} \cdot \tfrac{1}{3} \int_{|\xi|=1} |\xi|^{-3} d \omega(\xi) = \tfrac{1}{6 \pi^2} \cdot \tfrac{1}{z-3}
\]
\item[$k=1$:] near $z_1=3-1=2$, the function $\zeta(z)$ is analytic.
\item[$k=2$:] near $z_2=3-2=1$, we obtain
\[
\zeta(z) = \tfrac{(2\pi)^{-3}}{-1+z} \cdot \tfrac{1}{3} \int_{|\xi|=1} (-\tfrac{1}{2}) \cdot 1 \cdot |\xi|^{-3} d \omega(\xi) = -\tfrac{1}{12 \pi^2} \cdot \tfrac{1}{z-1}
\]
\item[$k=3$:] near $z_3=3-3=0$, the function $\zeta(z)$ is analytic.
\item[$k=4$:] near $z_4=3-4=-1$ we find that
\[
\zeta(z) = \tfrac{(2\pi)^{-3}}{1+z} \cdot\tfrac{1}{3} \int_{|\xi|=1} \tfrac{1}{8}\left[(-1)^2+2 \cdot (-1)\right] |\xi|^{-3} d \omega(\xi) = -\tfrac{1}{48 \pi^2} \cdot \tfrac{1}{z-(-1)}.
\]
\end{enumerate}

\textbf{Heat trace expansion.} Explicit computations as above yield the following
\begin{align*}
c_0 & = \res_{z=3} \Gamma(z) \zeta(z) = \Gamma(3) \cdot \tfrac{1}{6 \pi^2} =\tfrac{1}{3 \pi^2} \\
c_1 & = 0 \\
c_2 & = \res_{z=1} \Gamma(z) \zeta(z) = \Gamma(1) \cdot \left(-\tfrac{1}{12 \pi^2}\right) = -\tfrac{1}{12 \pi^2} \\
c_3 & = \res_{z=0} \Gamma(z) \zeta(z) = 1 \cdot \zeta(0) = \zeta(0) \\
c_4 & = 0.
\end{align*}
Also, $\tilde c_0 = \tilde{c}_1 = \tilde{c}_2 = \tilde{c}_3 =0$ and due to the double pole of $\Gamma(z) \zeta(z)$ at $z_4=-1$ 
\[
\tilde{c}_4 = \res_{z=-1} \Gamma(z) \zeta(z) = \tfrac{(-1)^1}{1!} \cdot \left(-\tfrac{1}{48 \pi^2}\right)=\tfrac{1}{48 \pi^2}
\]
The desired heat trace expansion for $A$ reads
\[
TR\left(e^{-t A} \right) \sim \tfrac{1}{3 \pi^2} e^t t^{-3} -\tfrac{1}{12 \pi^2} e^t t^{-1} + \zeta(0) e^t - \tfrac{1}{48 \pi^2} e^t t \log t   + \cdots,
\]
for the relativistic stable L\'evy process in $\R^3$. 

\begin{remark} \label{rem_banuelos_trace}
We remark that \cite{banuelos2014two} and \cite{park2014trace} consider a related case when an $\alpha$-stable relativistic  process moves in a bounded region $D$ in Euclidean space of dimension $n$. The authors assume different degrees of regularity of the boundary $\partial D$. In our notation, the authors consider the subordinator $f(\lambda)=(\lambda+c^{1/\alpha})^\alpha-c$.

For $c=1$ the two lowest orders of the heat trace are given in \cite{banuelos2014two} by
\begin{equation} \label{eq_asymptotics_banuelos}
\vol(D) C_1(t) e^t t^{-n/2\alpha} -C_2(t) \vol(\partial D),
\end{equation}
and in \cite{park2014trace} by
\begin{equation} \label{eq_asymptotics_song}
\vol(D) C_1' \left(\sum_{k=0}^l \tfrac{1}{k!} t^k\right) t^{-n/2\alpha} -C_2' \vol(\partial D) t^{-(n-1)/2\alpha},
\end{equation}
where $l$ is the largest integer less than $1/\alpha$. Ein either case, $\vol(D)$ the volume of $D$ with respect to Lebesgue measure and $\vol(\partial D)$ the surface area. 

In \eqref{eq_asymptotics_banuelos}, $C_1$ and $C_2$ are functions such that 
\begin{equation} \label{eq_banuelos_trace}
C_1(t) \to \tfrac{\Omega_n \Gamma(n/2\alpha)}{(2\pi)^n 2\alpha}
\end{equation}
as $t \to 0$ and $C_2(t)$ is bounded in the order of $e^{2t} t^{-(n-1)/2\alpha}$. Both $C_1$ and $C_2$ can be expressed in terms of probability densities and first exit times, respectively. In \eqref{eq_asymptotics_song}, $C_1'$ is a constant given by $C_1'=\frac{\Omega_n \Gamma(n/2\alpha)}{(2\pi)^n 2\alpha}$. The constant $C_2'$ reflects the boundary geometry of $D$.

Note that the limit in \eqref{eq_banuelos_trace} (and hence $C_1'$) is exactly the coefficient $c_0$ in our heat trace expansion \eqref{eq_lowest_coeffs_heat_trace} for the given subordinator. 
\end{remark}

\textbf{Acknowledgement.} The author thanks D. Applebaum and R. Ba{\~n}uelos for several helpful exchanges. The author is also indebted to E. Schrohe and S. Weber for cogent remarks on an earlier version of the manuscript and for fruitful discussions.

%%% ------------------------------------------------------------------------
%------------------------------------------------------- BIBLIOGRAPHY 
\bibliographystyle{plain}
\bibliography{Bibliography}

\end{document}